\documentclass[11pt]{amsart}
\usepackage{import}
\usepackage{mathabx}
\usepackage{bbm}
\usepackage[T1]{fontenc}
\usepackage{amsfonts}
\usepackage{mathrsfs}
\usepackage{stmaryrd}
\usepackage{upgreek}
\usepackage{textcomp}
\usepackage[mathscr]{eucal}

\usepackage{extarrow}
\usepackage{enumerate}
\usepackage{enumitem}
\usepackage{graphicx}
\usepackage[colorlinks=true, linkcolor=blue]{hyperref}
\usepackage{mathtools}
\usepackage{mathtext}
\usepackage{amsmath}
\usepackage{amssymb}
\usepackage{amsthm}
\usepackage{tikz}
\usepackage[all,cmtip]{xy}
\usepackage{comment}

\usepackage[a4paper,margin=1.2in]{geometry}

\usetikzlibrary{arrows}
\usetikzlibrary{matrix}
\usetikzlibrary{shapes}
\usetikzlibrary{snakes}
\usetikzlibrary{matrix}

\DeclareMathOperator{\et}{\textup{et}}

\global\long\def\set{(\textup{Sets})}

\newcommand{\Aut}{{\rm Aut}}

\newcommand{\Cov}{\textup{Cov}}

\renewcommand{\et}{\textup{\'et}}

\newcommand{\Noohi}{\textup{Noohi}}

\newcommand{\pet}{\textup{proét}}

\newcommand{\red}{{\rm red}}
\newcommand{\Sets}{\textup{Sets}}

\usepackage{bbm}
\usepackage{biblatex}
\usepackage{braket}
\theoremstyle{plain}
\newtheorem{thm}{Theorem}[section]
\newtheorem*{thm*}{Theorem}
\newtheorem{lem}[thm]{Lemma}
\newtheorem{cor}[thm]{Corollary}
\newtheorem{prop}[thm]{Proposition}

\theoremstyle{remark}
\newtheorem{rmk}[thm]{Remark}
\newtheorem{ex}[thm]{Example}

\theoremstyle{definition}
\newtheorem{defn}[thm]{Definition}

\newtheorem{sett}[thm]{Setting}

\theoremstyle{plain}
\newtheorem{thmI}{Theorem}

\newtheorem{thmII}{Theorem}

\newcommand{\triarrows}{\mathrel{\,\begin{tikzpicture}[baseline=-0.5ex, line width=0.4pt]
  \draw[->] (0,0.15) -- (1,0.15);
  \draw[->] (0,0) -- (1,0);
  \draw[->] (0,-0.15) -- (1,-0.15);
\end{tikzpicture}}}
\def\Cov{\operatorname{Cov}}
\def\Sets{\operatorname{Sets}}

\begin{document}

\title{The pro-\'etale fundamental group of singular schemes}

\author{Jiu-Kang Yu \and Lei Zhang}

\address{Jiu-Kang YU\\ Hetao Institute of Mathematics and Interdisciplinary
Sciences\\ Shenzhen, Guangdong Province\\ China}
\email{jiukangyu@himis-sz.cn}

\address{Lei ZHANG\\ Sun Yat-Sen University\\ School of Mathematics (Zhuhai)\\ Zhuhai, Guangdong Province\\ China}
\email{cumt559@gmail.com} 

\address{Marcin LARA\\ Instytut Matematyczny PAN\\ Śniadeckich 8\\ Warsaw, Poland}
\email{marcin.lara@impan.pl}

\date{\today}

\maketitle

\vspace{-1.3em}
\begin{center}
    \normalfont\textit{With an appendix by Marcin Lara}
\end{center}

\begin{abstract}
  We compute the pro-\'etale fundamental group of a connected Nagata
  J-2 scheme in terms of the \'etale fundamental groups of the
  normalizations of its irreducible components and a discrete free
  group.  The result generalizes a formula of E.~Lavanda for
  semi-stable curves and relies on a combination of proper descent
  techniques for \'etale morphisms and a combinatorial van Kampen
  construction for Noohi groups.  As a by‑product we characterize when
  a continuous representation of the pro‑\'etale fundamental group
  factors through a discrete quotient.
\end{abstract}

\section*{Introduction}

For a connected locally topologically Noetherian scheme \(X\) the
\textit{pro-\'etale fundamental group} \(\pi_1^\pet(X)\) was introduced
by Bhatt and Scholze in \cite{BS15}.  It refines Grothendieck’s \'etale
fundamental group \(\pi_1^\et(X)\) and classifies geometric covers of
\(X\).  While the two groups coincide for normal schemes, singular
schemes carry additional pro‑\'etale information: for a degenerate
curve over an algebraically closed field, \(\pi_1^\pet(X)\) is a
discrete free group, a phenomenon that appeared implicitly in the
theory of Mumford curves \cite{Mumford72}.

In \cite[Theorem~1.17]{Elena18} E.~Lavanda computed \(\pi_1^\pet(X)\)
for a semi‑stable curve over an algebraically closed field.  She
showed that the group is the Noohi coproduct of the \'etale fundamental
groups of the normalizations of its irreducible components and a free
group of an explicitly determined rank.  Our first main result extends
this formula to a much larger class of singular schemes.

\begin{thmI}\label{thmI}
  Let \(X\) be a connected Nagata J‑2 scheme, let \(\tilde X_i\) be the
  normalizations of its irreducible components, and let \(Z\) be the
  singular locus of $X_\red$.  Then \(\pi_1^\pet(X)\) can be expressed as an
  iterated Noohi coproduct and Noohi quotient from the \'etale
  fundamental groups \(\pi_1^\et(\tilde X_i)\), the pro‑\'etale
  fundamental groups of the connected components of \(Z\), and a
  discrete free group whose rank depends only on the combinatorics of
  the singularities.
\end{thmI}

The precise description is given in Theorems~\ref{thm:main-structure}
and~\ref{thm:disconnected-structure}.  When the singular locus is
particularly simple, the formula becomes completely explicit (cf.~Theorem
\ref{thm:curve-case}):

\begin{thmII}\label{thmII}
  Suppose that the singular locus \(Z\) of the reduced scheme
  \(X_{\mathrm{red}}\) is a disjoint union of spectra of separably
  closed fields (for instance, when \(X\) is a curve over a separably
  closed field).  Let \(n\) be the number of irreducible components of
  \(X\), let \(\tilde X_i\) be their normalizations, and let \(m\)
  (resp. \(\tilde m\)) be the number of connected components of \(Z\)
  (resp. of \(Z\times_X\tilde X\)).  Then there is an isomorphism of
  Noohi groups
  \[
    \pi_1^\pet(X,x)\;\simeq\;
    \Bigl(\coprod_{i=1}^n \pi_1^\et(\tilde X_i,x_i)\Bigr)
    \;\coprod\; \mathbb{F}_{\tilde m - m - n + 1},
  \]
  where \(\mathbb{F}_r\) denotes the free discrete group of rank \(r\)
  and the coproduct is taken in the category of Noohi groups.
\end{thmII}

Theorem~\ref{thmII} recovers Lavanda’s formula for semi‑stable curves as a
special case and generalises it to arbitrary curves as well as to
higher‑dimensional schemes whose singularities are isolated.

Our method relies on two geometric tools:
(1) the proper descent theorem for geometric covers due to
    Grothendieck and Rydh (Lemma~\ref{closed van kampen v2}), which
    allows one to glue geometric covers along closed subschemes;
    (2) a van Kampen machine for Noohi groups (Proposition~\ref{van kampen
    machine})
    that translates the geometric gluing into an algebraic presentation
    of the fundamental group.
This approach does not depend on the
characteristic, and reduces the computation of \(\pi_1^\pet(X)\) to an
induction on the number of connected components and the dimension of
the singular locus.  The result is a description of the pro‑\'etale
fundamental group entirely in terms of the much better understood
\'etale fundamental groups of normal schemes and discrete free groups.

%As a further application of the structure theorems we characterise
%when a continuous representation of \(\pi_1^\pet(X)\) factors through a
%discrete quotient (Proposition~\ref{prop:discrete-quotient}).

\subsection*{Notations and Conventions}

\begin{enumerate}
 % \item For a topological group \(\pi\) and a field \(K\),
 %   \(\Rep_K^\cts(\pi)\) denotes the category of finite dimensional
 %   continuous \(K\)‑representations, where \(K\) is viewed as a
 %   discrete field.
  \item For a scheme \(X\), \(\Cov(X)\) denotes the category of
    \textit{geometric covers} of \(X\) in the sense of
    \cite[Definition~7.3.1]{BS15}.  If \(X\) is locally topologically
    Noetherian and connected, \(\pi_1^\pet(X,x)\) is the corresponding
    Noohi group.
  \item All schemes are assumed to be locally topologically
    Noetherian.  A \textit{Nagata} scheme is a Noetherian scheme in
    which every integral closed subscheme has a finite integral
    closure \cite[\href{https://stacks.math.columbia.edu/tag/033Y}{033Y}]{stacks-project}.
    A scheme is \textit{J‑2} if its singular locus is closed
    \cite[\href{https://stacks.math.columbia.edu/tag/07R2}{07R2}]{stacks-project}.
\end{enumerate}

\section{Noohi groups}

\subsection{Results about Noohi groups}
Let \(G\) be a topological group (not necessarily Hausdorff) and let
\(G\text{-}\Sets\) be the category of discrete sets equipped with a
continuous \(G\)-action.  Denote by
\(
  F_G\colon G\text{-}\Sets\longrightarrow\Sets
\)
the forgetful functor.

\begin{defn}\label{def:noohi}
  A topological group \(G\) is called a \textit{Noohi group} if the natural map
\(G\to\Aut(F_G)\) is an isomorphism of topological groups, where
\(\Aut(F_G)\) is equipped with the compact‑open topology.
\end{defn}

\begin{lem}\label{Small facts about Noohi groups}
  A Noohi group \(\pi\) is Hausdorff and has a basis of open
  neighbourhoods of the identity consisting of subgroups.
\end{lem}
\begin{proof}
  Since \(\pi\) is Noohi, \(\pi \cong \Aut(F_\pi)\) as topological
  groups, where \(F_\pi\colon\pi\text{-}(\Sets)\to(\Sets)\) is the
  forgetful functor.  Let \(T\) be the set of open subgroups of
  \(\pi\).  For each \(U \in T\) the discrete \(\pi\)-set \(\pi/U\)
  yields a restriction map \(\Aut(F_\pi) \to \Aut(\pi/U)\).  These
  maps embed \(\Aut(F_\pi)\) as a topological subgroup of
  \(\prod_{U \in T} \Aut(\pi/U)\) (with the compact‑open topology on
  each factor).  Each \(\Aut(\pi/U)\) is Hausdorff and has a basis of
  open subgroups; these properties are inherited by arbitrary products
  and by topological subgroups.  Hence \(\pi \cong \Aut(F_\pi)\) is
  Hausdorff and admits a neighbourhood basis of open subgroups.
\end{proof}

\begin{lem}\label{faithfulness of Tannakian}
  Let \(f,g\colon G\to\pi\) be continuous homomorphisms of topological
  groups.  Assume that \(\pi\) is Hausdorff and has a basis of open
  neighbourhoods of the identity consisting of subgroups.  If the
  induced functors \(f^*,g^*\colon\pi\text{-}(\Sets)\to G\text{-}(\Sets)\)
  coincide, then \(f=g\).
\end{lem}
\begin{proof}
  The intersection of all open subgroups of \(\pi\) is \(\{e\}\); hence
  it suffices to test equality after the quotient maps
  \(\pi\to\pi/U\) for every open subgroup \(U\).  Since
  \(\pi/U\in\pi\text{-}(\Sets)\), the coincidence of \(f^*\) and
  \(g^*\) forces these compositions to be equal, so \(f=g\).
\end{proof}

The following lemma generalises \cite[§7.2]{BS15}, \cite[Lemma~2.51]{Lara19} to
topological groups that are not necessarily Hausdorff.

\begin{lem}\label{lem:noohi-adjoint}
  The inclusion functor
  \[
    (\text{Noohi groups})\;\hookrightarrow\;(\text{topological groups})
  \]
  admits a left adjoint, which we denote by \((-)^{\Noohi}\).
\end{lem}
\begin{proof}
  Let \(G\) be a topological group.  The category \(G\text{-}(\Sets)\)
  is a tame infinite Galois category.  Let
  \(F\colon G\text{-}(\Sets)\to(\Sets)\) be the forgetful functor and
  set \(G^{\Noohi}:=\Aut(F)\), equipped with the compact‑open topology.
  By \cite[Theorem~7.2.5.1]{BS15} the group \(G^{\Noohi}\) is Noohi,
  and by \cite[Theorem~7.2.5.3]{BS15} the natural map
  \(\alpha\colon G\to G^{\Noohi}\) induces an equivalence
  \(G^{\Noohi}\text{-}(\Sets)\simeq G\text{-}(\Sets)\).

  Let \(\gamma\colon G\to G'\) be a continuous homomorphism with
  \(G'\) a Noohi group.  Then \(G'\) is Hausdorff and has a basis of open subgroups  by
  Lemma~\ref{Small facts about Noohi groups}.  The functor
  \(G'\text{-}(\Sets)\to G\text{-}(\Sets)\) induced by \(\gamma\)
  corresponds, via \cite[Theorem~7.2.5.2]{BS15}, to a continuous
  homomorphism \(\beta\colon G^{\Noohi}\to G'\).  Lemma~\ref{faithfulness of
  Tannakian}
  implies \(\beta\circ\alpha=\gamma\); uniqueness of \(\beta\) follows
  from the same lemma.  Hence \((-)^{\Noohi}\) is a left adjoint to the
  inclusion.
\end{proof}

\begin{defn}\label{def:noohi-quotient}
  Let \(\pi\) be a Noohi group and \(H\subseteq\pi\) a subgroup.
  Denote by \(\langle H\rangle\) the normal subgroup generated by
  \(H\).  The \textit{Noohi quotient of \(\pi\) by \(H\)} is
  \((\pi/\langle H\rangle)^\Noohi\). It corresponds
  to the full subcategory of $\pi\text{-}(\Sets)$ consisting of 
  objects on which $H$ acts trivially. If \(\{f_i,g_i\}_{i\in I}\) are elements of \(\pi\), we often form the
Noohi quotient \((\pi/\langle H\rangle)^\Noohi\) with \(H\) being the
subgroup generated by \(\{f_i^{-1}g_i\}_{i\in I}\).  We refer to this
as the {\it Noohi quotient of \(\pi\) by the relations
  \(\{f_i=g_i\}_{i\in I}\).}\end{defn}

\begin{lem}\label{lem:noohi-coproduct}
  Finite fibred coproducts exist in the category of Noohi groups.
\end{lem}
\begin{proof}
  For a diagram \(\pi_1\xleftarrow{p}\pi\xrightarrow{q}\pi_2\), the
  fibred coproduct \(\pi_1\coprod^\pi\pi_2\) is obtained as the Noohi
  quotient of \(\pi_1\coprod\pi_2\) (cf.~\cite[Example 7.2.6]{BS15}) by the relations
  \(\{p(x)=q(x)\}_{x\in\pi}\). 

Alternatively, it can also be
    defined by the Noohi group corresponding to the tame infinite
    Galois category of discrete sets equipped
with two  actions from $\pi_1$ and $\pi_2$ which
agree on $\pi$.
\end{proof}

\subsection{A van Kampen construction for Noohi groups}

The construction we are about to discuss, in the case of discrete
groups, can be found in \cite[8.4.1, p.~333]{Br06}, where it is attributed
to van Kampen.
Let \(s\geq 1\) and let \(\pi,\pi ',\pi_1'',\ldots,\pi_s''\) be
Noohi groups.  Let \(\psi_i: \pi_i''\to \pi\), \(\phi_i:\pi_i''\to \pi
'\) be continuous homomorphisms.  We define a
(discrete) group \(F\) by generators  \(\{u_{ij}: 1\leq
i,j\leq s\}\) and relations \(u_{ii}=e\), \(u_{ij}u_{jk}=u_{ik}\).
Evidently, \(F\) is free of rank \(s-1\) on the generators
\(v_2,\ldots,v_s\), where \(v_j\coloneqq u_{1j}\) for \(j=2,\ldots,s\).  We
put \(v_1=e\).

\begin{lem}\label{vk-cons-1}
  The following Noohi groups are isomorphic:
  \begin{itemize}
    \item[{\rm(i)}] The Noohi coproduct of \(\pi '\) and \(F\).
    \item[{\rm (ii)}] The Noohi coproduct of \(s\) copies of \(\pi '\),
      and one copy of \(F\), then Noohi
      quotient by the relations \(u_{ij}^{-1}[y]_iu_{ij}=[y]_j\),
      \(1\leq i,j\leq s\), \(y \in \pi '\).  Here \([y]_i\) denote \(y\)
      regarded as an element in the \(i\)-th copy of \(\pi '\).
    \end{itemize}
  \end{lem}

  \begin{proof}  The isomorphism is characterized by \(y
    \leftrightarrow [y]_1\) for \(y \in \pi '\) and \(g
    \leftrightarrow g\) for \(g \in F\).  The universal property of
    coproduct implies that there are homomorphisms in both directions
    matching these elements, and they are inverse to each other.
  \end{proof}

  \begin{lem}\label{vk-cons-2}
    The following Noohi groups are isomorphic:
  \begin{itemize}
    \item[{\rm (i)}] The Noohi coproduct of \(\pi\) and the group in
      \ref{vk-cons-1} {\rm (i)}, Noohi quotient by the relations
      \(\psi_i(a)=v_i^{-1}\phi_i(a)v_i\), \(a \in \pi ''_i\),
      \(i=1,\ldots,s\).
  \item[{\rm(ii)}] The Noohi coproduct of \(\pi\) and the group
    in \ref{vk-cons-1} {\rm (ii)}, Noohi quotient by the relations
    \(\psi_i(a)=[\phi_i(a)]_i\), \(a \in \pi ''_i\), \(i=1,\ldots,s\).
    \item[{\rm (iii)}] The Noohi coproduct of \(F\) with
      \(\pi\coprod\limits^{\pi_1'',\phi_1,\psi_1}\pi '\),
      Noohi quotient by the relations
      \(\psi_i(a)=v_i^{-1}\phi_i(a)v_i\), \(a\in \pi ''_i\), \(i=2,\ldots,s\).
      \item[{\rm (iv)}] The Noohi coproduct of \(F\) and \((\)the Noohi fiber coproduct of
        \(\pi\to\pi\coprod\limits^{\pi_1'',\phi_1,\psi_1}\pi ',\ldots,\pi\to \pi\coprod\limits^{\pi_s'',\phi_s,\psi_s}\pi ')\), Noohi quotient by the
        relations \(u_{ij}^{-1}(e*_iy)u_{ij}=e*_jy\) for all
\(y \in \pi '\), \(i,j=1,\ldots,s\).  Here \(e*_i y\) denote
the product of \(e \in \pi\) and \(y\in \pi '\) in \(\pi\coprod\limits^{\pi_i'',\phi_i,\psi_i}\pi '\).
    \end{itemize}
\end{lem}

\begin{proof}  The isomorphism between (i) and (ii) is induced by that
  of Lemma \ref{vk-cons-1}.  The isomorphism between (i) and (iii) is rather
  obvious. The isomorphism between (iii) and (iv) is constructed in a
  way similar to that of Lemma \ref{vk-cons-1}.
\end{proof}

 We will denote the group in Lemma \ref{vk-cons-2} as \[{\bf
    VK}(\pi,\pi
';\pi_1'',\ldots,\pi_s'')_{(\psi_1,\ldots,\psi_s),(\phi_1,\ldots,\phi_s)}\]
  and omit the homomorphisms in subscripts when there is no
  confusion.  The description (i) is the one usually found in the
  literature, e.g.~\cite[8.4.1, p.~333]{Br06}.  It relies on singling
  out the index \(1\) and so does (iii).  The descriptions (ii)
  and (iv) make it clear that the construction actually treats all
  indices in equal footing.  This construction is very useful due to
  the next result.

\def\scrC{\mathscr{C}}
\def\scrD{\mathscr{D}}
\def\scrE{\mathscr{E}}
\begin{sett}
  To state the result, we work in the following situation with
  \(s\geq 1\):
  \begin{itemize}
    \item Let \((\scrC,F)\), \((\scrD,G)\),
  \((\scrE_1,H_1),\ldots,(\scrE_s,H_s)\) be tame infinite Galois
  categories.
  \item For \(j=1,\ldots,s\), let
  \(u_j:\scrC\to \scrE_j\) (resp.~\(v_j:\scrD\to \scrE_j\)) be  a functor
 such that \((\scrC,H_j\circ u_j)\) (resp.~\((\scrD,H_j\circ v_j)\) is a tame infinite Galois
 category, and \(H_j\circ u_j\cong F\)
 (resp.~\(H_j\circ v_j\cong G\)).
 \item Recall that an object of the 2-fibre product
     (\cite[\href{https:https://stacks.math.columbia.edu/tag/003R}{003R}]{stacks-project})
\[
\scrC \mathop{\times}\limits_{\scrE_1\times\cdots\times \scrE_s} \scrD
\]
is a triple \((X,Y,\phi)\), where \(X \in \scrC\), \(Y \in \scrD\),
and \(\phi\) is an isomorphism 
\[
    \bigl(u_j(X)\bigr)_{1\leq j\leq
s}\to\bigl(v_j(Y)\bigr)_{1\leq j\leq s}
\]
in $\scrE_1\times\cdots\times \scrE_s$. We define a functor \(F'\) from this
2-fibre product to \(({\rm Sets})\) by setting \(F'(X,Y,\phi)=F(X)\).
\item For \(j=1,\ldots,s\), choose an isomorphism of fiber functors on
  \(\scrC\colon H_j\circ u_j \xrightarrow{\cong} F\) and denote the
  composition \(\pi_1(\scrE_j,H_j)\to \pi_1(\scrC,H_j\circ u_j) \to
  \pi_1(\scrC,F)\) by \(\psi_j\).  We also choose an isomorphism of fiber functors on
  \(\scrD\colon H_j\circ v_j \xrightarrow{\cong} G\) and denote the
  composition \(\pi_1(\scrE_j,H_j)\to \pi_1(\scrD,H_j\circ v_j) \to
  \pi_1(\scrD,G)\) by \(\phi_j\).
\end{itemize}
\end{sett}

\begin{prop} \label{van kampen machine}
  With the above setting, the pair
\[
(\scrC \mathop{\times}\limits_{\scrE_1\times\cdots\times \scrE_s} \scrD,F')
\]
is a tame infinite Galois category, whose fundamental group is
isomorphic  to
\[
{\bf VK}(\pi_1(\scrC,F),\pi_1(\scrD,G);\pi_1(\scrE_1,H_1),\ldots,\pi_1(\scrE_s,H_s))_{(\psi_1,\ldots,\psi_s),(\phi_1,\ldots,\phi_s)}.
\]

% to the Noohi fibre coproduct
% \[
%  \left( \pi_1(\scrC,F)\mathop{\textstyle\coprod}\limits_{\pi_1(\scrE_1,H_1),\psi_1,\phi_1}\pi_1(\scrD,G)\right)\mathop{\textstyle\coprod}\coprod_{j=2}^s\mathbb{Z}_j,
% \]
% where each \(\mathbb{Z}_j\) is a free group of rank \(1\) with generator
% \(1_j\), Noohi quotient by the relations
% \[
%   \phi_j(a)\cdot 1_j=1_j\cdot \psi_j(a),\hspace{5pt} a \in
%   \pi_1(\scrE_j,H_j),\hspace{5pt} j=2,\ldots,s.
% \]
\end{prop}

\begin{proof}
  We proceed in three steps.

  \textbf{Step 1: Reduction to group actions.}
  By \cite[Theorem~7.2.5]{BS15} every tame infinite Galois category
  with a fiber functor is equivalent to the category of discrete sets
  with a continuous action of its fundamental group.  Applying this to
  $\scrC,\scrD$ and the $\scrE_j$, we may assume without loss of
  generality that
  \[
    \scrC \simeq \pi_{\scrC}\text{-}\Sets,\quad
    \scrD \simeq \pi_{\scrD}\text{-}\Sets,\quad
    \scrE_j \simeq \pi_j\text{-}\Sets,
  \]
  where $\pi_{\scrC} = \pi_1(\scrC,F)$, $\pi_{\scrD} = \pi_1(\scrD,G)$,
  and $\pi_j = \pi_1(\scrE_j,H_j)$.  The functors $u_j$ and $v_j$, as
  well as the chosen compatibilities of fiber functors, translate
  through Tannakian duality into the continuous homomorphisms
  $\psi_j : \pi_j \to \pi_{\scrC}$ and $\phi_j : \pi_j \to \pi_{\scrD}$
  described in Setting~2.8.  The fiber functor $F'$ on the fibred
  product becomes the forgetful functor to $\Sets$.

  \textbf{Step 2: Unpacking the fibred product.}
  An object of the 2‑fibre product
  $\scrC' = \scrC \times_{\prod_j \scrE_j} \scrD$
  is a triple $(A,B,\{\lambda_j\}_{j=1}^s)$ where $A$ is a
  $\pi_{\scrC}$-set, $B$ is a $\pi_{\scrD}$-set, and
  $\lambda_j : A \xrightarrow{\cong} B$ are bijections of $\pi_j$-sets,
  with $A$ viewed as a $\pi_j$-set via $\psi_j$ and $B$ via $\phi_j$.
  We fix the index~$1$ and use $\lambda_1$ to identify $B$ with $A$.
  Then $A$ carries commuting actions of $\pi_{\scrC}$ and $\pi_{\scrD}$
  that are intertwined along $\pi_1$ through the relation
  $\psi_1(a) = \phi_1(a)$ for $a \in \pi_1$.  This precisely means that
  $A$ is a discrete set with a continuous action of the Noohi fibred
  coproduct
  \[
    \Pi_1 \;:=\; \pi_{\scrC}
    \mathop{\textstyle\coprod}\limits_{\pi_1,\,\psi_1,\,\phi_1}
    \pi_{\scrD}.
  \]

  The remaining isomorphisms $\lambda_j$ ($j\ge 2$) become bijections
  $\lambda_j : A \to A$ that satisfy the equivariance condition
  \[
    \phi_j(a) \circ \lambda_j \;=\; \lambda_j \circ \psi_j(a)
    \qquad (a\in \pi_j).
  \]
  We regard each $\lambda_j$ as an automorphism of the underlying
  set $A$ and let $F$ be the free discrete group generated by symbols
  $v_2,\dots,v_s$ (with $v_1 = e$) that act on $A$ via
  $v_j \cdot x := \lambda_j(x)$.  The condition above rewrites as
  \[
    \psi_j(a) \;=\; v_j^{-1}\,\phi_j(a)\,v_j
    \quad\text{in the group of automorphisms of }A,
  \]
  where $\phi_j(a)$ is viewed as an element of $\pi_{\scrD}$ and hence
  of $\Pi_1$.

  \textbf{Step 3: Identification of the fundamental group.}
  The data $(A,\{\lambda_j\}_{j\ge 2})$ is exactly the same as a
  discrete set with an action of the group
  \[
    \Pi \;:=\;
    \Bigl(
      \Pi_1 \coprod F
    \Bigr)
    \Big/ \big\langle
      \psi_j(a) = v_j^{-1}\phi_j(a)v_j
      \;\big|\; a\in\pi_j,\; j=2,\dots,s
    \big\rangle^{\Noohi},
  \]
  where the relations are imposed in the sense of
  Definition~\ref{def:noohi-quotient}.  By Lemma~\ref{vk-cons-2}~(iii)
  this group is precisely
  $\mathbf{VK}(\pi_{\scrC},\pi_{\scrD};\pi_1,\dots,\pi_s)_{(\psi),(\phi)}$.
  The free group $F$ accounts for the additional identifications
  $\lambda_2,\dots,\lambda_s$ after the first index has been used to
  identify the two sides.

  Thus the pair $(\scrC',F')$ is equivalent to the category of
  discrete $\Pi$-sets with the forgetful functor.  By
  \cite[Example~7.2.2]{BS15} the fundamental group of
  $(\scrC',F')$ is isomorphic to $\Pi$, which completes the proof.
\end{proof}

%\begin{rmk} With the notation of the proof, it is obvious that under the identification
%  \(\pi_1(\scrC ',F')\simeq \pi\), the maps \(\pi_1(\scrC ,F) \to
%  \pi_1(\scrC ',F')\), \(\pi_1(\scrD,G)) \to \pi_1(\scrC ',F')\)
%  correspond to the canonical maps \(\pi_1(\scrC,F)\to \pi\),
%  \(\pi_1(\scrD,G)\to \pi\).
%\end{rmk}

\begin{rmk}
More sophisticated forms of this sort of result can be found in
\cite[Corollary 5.3, p.~19]{stix2006} and \cite[Corollary 3.18,
p.~32]{Lara19}.  However, the form given here is more readily
applicable and is enough for almost all known applications, such as
\cite[Theorem~1.17]{Elena18}.  More applications are given in the next few subsections.
%  Ideas along the same lines can be found 
%  in \cite[8.4.1, p.~333]{Br06} which are documented in a great
%     generality. They have been used by E. Lavanda in \cite[Theorem
%     1.17, p.~33]{Elena18}. Similar ideas have also been used in
%     \cite[Corollary 5.3, p.~19]{stix2006} and \cite[Corollary 3.18, p.~32]{Lara19} to build up van Kampen style theorems. We will refer
% to Proposition \ref{van kampen machine}  as the \textit{combinatorial folklore}.
\end{rmk}

\section{The pro‑\'etale fundamental group}

\subsection{Van Kampen theorems for geometric covers}
\label{pass to reduced structure} From now on we assume that $X$ is a locally topologically
    Noetherian scheme. Let $\Cov(X)$ denote the category of \textit{geometric covers of
    $X$} in the sense of \cite[Definition 7.3.1]{BS15}.  This
  construction gives a category fibered over the category of locally
  topological Noetherian schemes in an obivious way.
  
  Let $X_\red$
    denote the reduced induced scheme structure of $X$. Then by
    \cite[Exposé VIII, Théorème 1.1, p.~247]{SGA4} or \cite[Exposé IX, 4.10,
    p.~186]{SGA1} we can conclude that
    $\Cov(X)\xrightarrow{\simeq}\Cov(X_\red)$ just as in
    \cite[Lemma 1.15]{Elena18}. Thus if $X$ is a locally
    Noetherian connected scheme and $x\in X$ is a geometric point, then
    $\pi_1^\pet(X_\red,x)\to\pi_1^\pet(X,x)$ is an isomorphism.

\begin{lem}\label{general van kampen}
    Let $f_1:Z_1\rightarrow X, f_2: Z_2\to X$ be monomorphisms
    of
    schemes.  Assume that the induced map
    \(Z_1\coprod
Z_2\to X\) is a morphism of effective descent for  \(\Cov(-)\).  
  Then we have an equivalence given by the pullback functor:
  \[\Cov(X)\simeq \Cov(Z_1)\times_{\Cov(Z)}\Cov(Z_2)\]
  where $Z=  Z_1\times_XZ_2$. 
\end{lem}

\begin{proof} 
    %We refer to Stack Project 023B for
  %the notion of descent, and Stack Project 003R for the definition of
  %fiber products of categories.  
    By the effectiveness of descent, an
  object of \(\Cov(X)\) amounts to a descent datum, which consists of \(Y_i\in
  \Cov(Z_i)\), \(i=1,2\) together with \(\varphi_{ij}: p_1^* Y_i \to
  p_2^* Y_j\) for \(i,j \in \{1,2\}\) satisfying the cocycle
  condition.  The cocycle condition says that \(\varphi_{11}\) and \(\varphi_{22}\) are
  just identities (by the monomorphism assumption) and
  \(\varphi_{12},\varphi_{21}\) give
  isomorphisms \(Y_1|_Z \simeq Y_2|_Z\), inverse to each other.  Thus we get an object in the
  \(\Cov(Z_1) \times_{\Cov(Z)} \Cov(Z_2)\).  It is routine to verify
  that this construction gives an equivalence.
  % 023B is not the ideal reference.  Most references discuss this
  % in the context of a site.  But we only want to talk about descent
  % along a family of morphisms, and 023B did so, although
  % specifically for coherent sheaves.
\end{proof}

\begin{ex} \label{van kampen}
  The asssumption of Lemma~\ref{general van kampen} is satisfied when
  \(Z_1, Z_2\) are open subschemes of \(X\) such that
  \(X=Z_1\cup Z_2\).  This is due to the fact that geometric covers are locally
  constant sheaves in the pro-étale topology (cf.~\cite[Lemma 7.3.9]{BS15}), so in particular, they satisfy
  pro-étale descent.
\end{ex}

\begin{ex}\label{closed van kampen}
The asssumption of Lemma~\ref{general van kampen} is satisfied when
\(X\) is locally Notherian and
  \(Z_1, Z_2\) are closed subschemes of \(X\) such that
  \(X=Z_1\cup Z_2\) set-theoretically. This is a consequence of
  \cite[Proposition 1.16]{Elena18} which relies on Rydh's work \cite{Rydh10}
  generalizing \cite[IX Théorème 4.12]{SGA1}. 
\end{ex}

\begin{lem} \label{closed van kampen v2}
Let $X$ be a locally Noetherian scheme. Let $Z\subseteq X$ be a closed subscheme.
Consider a proper surjective map
\[\tilde{X}\xrightarrow{\hspace{5pt}f\hspace{5pt}}X\]
Denote $Z\times_X \tilde{X}$ by \(\tilde Z\). Suppose that the union
of  the image
of the two closed immersions
\[
  \Delta_f:\tilde X\to \tilde X\times_X\tilde X,\quad
  \tilde Z\times_Z \tilde Z \to \tilde X \times_X  \tilde X
\]
is \(\tilde X\times _X \tilde X\) set-theoretically.
Then the pullback functor induces an equivalence of categories:
\begin{equation}\label{eq:Ziel}\Cov(X)\xrightarrow{\cong} \Cov(\tilde{X})\times_{\Cov(\tilde
Z)}\Cov(Z)\end{equation}
\end{lem}

\begin{proof}
  Consider the Čech nerve of the proper surjective map \(f\).  Since
  \(\Cov(-)\) satisfies proper descent by
  \cite[Proposition~1.16]{Elena18}, the category \(\Cov(X)\) is
  equivalent to the \(2\)-limit (the category of descent data) of the
  truncated cosimplicial diagram
  \[
    \Cov(\tilde X) \rightrightarrows \Cov(\tilde X^{\times_X 2})
    \triarrows \Cov(\tilde X^{\times_X 3}).
  \]
  Analogously, \(\Cov(Z)\) is the \(2\)-limit of the cosimplicial
  diagram for the proper surjective map \(\tilde Z \to Z\):
  \[
    \Cov(\tilde Z) \rightrightarrows \Cov(\tilde Z^{\times_Z 2})
    \triarrows \Cov(\tilde Z^{\times_Z 3}).
  \]

  By Lemma~\ref{closed van kampen} applied to the closed cover of
  \(\tilde X^{\times_X 2}\) by \(\Delta_f(\tilde X)\) and
  \(\tilde Z^{\times_Z 2}\), we have a natural equivalence
  \begin{equation}\label{eq:double}
    \Cov(\tilde X^{\times_X 2})
    \xrightarrow{\;\sim\;}
    \Cov(\tilde X) \times_{\Cov(\tilde Z)} \Cov(\tilde Z^{\times_Z 2}).
  \end{equation}
  The same lemma applied to the triple product
  \(\tilde X^{\times_X 3}\) (covered by the triple diagonal and
  \(\tilde Z^{\times_Z 3}\)) yields a compatible equivalence
  \begin{equation}\label{eq:triple}
    \Cov(\tilde X^{\times_X 3})
    \xrightarrow{\;\sim\;}
    \Cov(\tilde X) \times_{\Cov(\tilde Z)} \Cov(\tilde Z^{\times_Z 3}).
  \end{equation}
  The face maps between the double and triple products are induced by
  pullbacks along the various projections, and the equivalences above
  intertwine these maps.

  Plugging \eqref{eq:double} and \eqref{eq:triple} into the
  cosimplicial diagram for \(X\) yields an equivalence of diagrams
  \begin{align*}
    &\bigl(\Cov(\tilde X) \rightrightarrows \Cov(\tilde X^{\times_X 2})
    \triarrows \Cov(\tilde X^{\times_X 3})\bigr)\\
      \;\simeq\;&
    \bigl(\Cov(\tilde X) \rightrightarrows
    \Cov(\tilde X)\times_{\Cov(\tilde Z)}\Cov(\tilde Z^{\times_Z 2})
    \triarrows
    \Cov(\tilde X)\times_{\Cov(\tilde Z)}\Cov(\tilde Z^{\times_Z 3})\bigr).
\end{align*}
  The second diagram is the product of the constant cosimplicial object
  \(\Cov(\tilde X)\) with the Čech nerve of \(\tilde Z \to Z\).  Limits
  commute with fiber products in this sense, hence
  \[
    \Cov(X) \;\simeq\;
    \Cov(\tilde X) \times_{\Cov(\tilde Z)}
    \bigl( \lim \Cov(\tilde Z^{\bullet+1}) \bigr).
  \]
  But \(\lim \Cov(\tilde Z^{\bullet+1})\) is exactly the category of
  descent data for \(\tilde Z \to Z\), which by proper descent is
  equivalent to \(\Cov(Z)\).  Thus
  \[
    \Cov(X) \;\simeq\; \Cov(\tilde X) \times_{\Cov(\tilde Z)} \Cov(Z),
  \]
  as required.
\end{proof}

\subsection{Schemes with connected singular locus}

Throughout this section and next ones, \(X\) denotes a connected Nagata J‑2 scheme.
Let \(X_1,\dots,X_n\) be its irreducible components.  Denote by
\(f\colon\tilde X\to X\) the normalisation of \(X_\red\) relative to
the disjoint union of the spectra of its generic points.  The map
\(f\) is finite and surjective.  Let \(Z\subseteq X\) be the singular
locus of \(X_\red\) equipped with the reduced induced structure, and
set \(\tilde Z:=\tilde X\times_X Z\).

For each \(i=1,\dots,n\) let \(\tilde X_i\) be the connected component
of \(\tilde X\) corresponding to \(X_i\).  Decompose
\(\tilde X_i\cap\tilde Z\) into its connected components
\(Z_{i1},\dots,Z_{in_i}\).

\begin{thm}\label{thm:structure-connected-covers}
  The pullback functor induces an equivalence
 \begin{center}
\resizebox{13cm}{!}{\parbox{\linewidth}{ \[
    \Cov(X)\xrightarrow{\;\simeq\;}
    \prod^{1\leq i\leq
n}_{\Cov(Z)}\,
    \Bigl(
      \Cov(\tilde X_i)
      \mathop{\times}\limits_{\prod_{j=1}^{n_i}\Cov(Z_{ij})}
      \Cov(Z)
    \Bigr)
\]}}\end{center}
\end{thm}
\begin{proof}  By Lemma \ref{closed van kampen v2} the  pullback map
 \begin{center}
\resizebox{13cm}{!}{\parbox{\linewidth}{
    \[
    \begin{aligned}\Cov(X)\xrightarrow{\ \simeq\ }&\left(\Cov(\coprod_{1\leq i\leq
                n}\tilde{X}_i)\mathop{\times}\limits_{\Cov(\coprod_{ij}Z_{ij})}\Cov(Z)\right)\\=&\left[\left(\prod_{1\leq i\leq
n}\Cov(\tilde{X}_i)\right)\mathop{\times}\limits_{\prod_{ij}\Cov(Z_{ij})}\Cov(Z)\right]\end{aligned}\]}}\end{center}
        is
an equivalence. Then one checks easily that the later is the
same as the right hand side in the statement of the theorem.
\end{proof}

Now choose geometric points \(x\in Z\) and \(x_{ij}\in Z_{ij}\).
Identifying the fibre functors \(F_x\) and \(F_{f(x_{ij})}\), we obtain
isomorphisms
\[
  \pi_1^\pet(Z,x)\simeq\pi_1^\pet(Z,f(x_{ij})).
\]
Let \(\phi_{ij}\colon\pi_1^\pet(Z_{ij},x_{ij})\to\pi_1^\pet(Z,x)\) be
the map induced by \(Z_{ij}\to Z\).  Similarly, choose paths inside
\(\tilde X_i\) to identify
\(\pi_1^\pet(\tilde X_i,x_{ij})\simeq\pi_1^\et(\tilde X_i,x_{i1})\),
and let \(\psi_{ij}\) be the composition
\(\pi_1^\pet(Z_{ij},x_{ij})\to\pi_1^\pet(\tilde X_i,x_{ij})\simeq
\pi_1^\et(\tilde X_i,x_{i1})\).

\begin{thm}\label{thm:main-structure}
  There is an isomorphism of Noohi groups
  \[\coprod_{1\leq i\leq
n}^{\pi_1^\pet(Z,x)}{\bf VK}\bigl(\pi_1^\et(\tilde{X}_i,x_{i1}),\pi_1^\pet(Z,x);\pi_1^\pet(Z_{i1},x_{i1}),\ldots,\pi_1^\pet(Z_{in_i},x_{in_i})\bigr)
\]
  Moreover, the canonical maps from \(\pi_1^\pet(Z,x)\) and
  \(\pi_1^\et(\tilde X_i,x_{i1})\) to \(\pi_1^\pet(X,x)\) correspond,
  via this isomorphism, to the natural inclusions into the coproduct.
\end{thm}
\begin{proof}
  This follows from Theorem~\ref{thm:structure-connected-covers} and
  Proposition~\ref{van kampen machine}.
\end{proof}

\subsection{Disconnected singularities: a d\'evissage}

Now drop the assumption that \(Z\) is connected.  Let
\(Z_1,\dots,Z_m\) be the connected components of \(Z\).  For each
\(j=1,\dots,m\) define the open subscheme
\[
  T_j:=\Bigl(\bigcup_{X_i\cap Z_j\neq\emptyset} X_i\Bigr)
  \setminus\Bigl(\bigcup_{t\neq j} Z_t\Bigr).
\]

\begin{lem}\label{lem:devissage}
  \begin{enumerate}
    \item Each \(T_j\) is connected and open in \(X\).
    \item \(X=T_1\cup\cdots\cup T_m\).
    \item \(Z_j\subset T_j\) and \(Z_j\cap T_t=\emptyset\) for \(t\neq j\).
    \item There exists an index \(k\) such that \(\bigcup_{t\neq k}T_t\) is connected.
  \end{enumerate}
\end{lem}
\begin{proof}  (i) Assume \(X_i \cap Z_j \neq\emptyset\) and \(X_s \cap
  Z_j=\emptyset\).  Then \(X_i \cap X_s\), being a subset of the
  singular locus, lies in \(\bigcup_{t\neq j} Z_t\).  Therefore,
  \(X_s\) is disjoint from \(T_j\).  Thus 
  \begin{center}
  \resizebox{13cm}{!}{\parbox{\linewidth}{      \begin{align*}
          T_j=&\left[\left(\bigcup_ {X_i\cap Z_j\ne
          \emptyset}X_i\right)\setminus \left(\bigcup_{X_s\cap Z_j=\emptyset}
    X_s\right)\right]\setminus\left(\bigcup_{t\neq
                      j}Z_t\right)\\=&\left[X\setminus \left(\bigcup_{X_t\cap
                      Z_j=\emptyset}X_t\right)\right]\setminus\left(\bigcup_{t\neq
    j}Z_t\right)
\end{align*}
}}\end{center}
is open in \(X\).  Since \(T_j\) is the union of the connected sets
\(Z_j\) and \(X_i\setminus
\bigl(\bigcup_{t\neq j} Z_t\bigr)\) (over those \(i\) such that
\(X_i\cap Z_j\neq\emptyset)\), and each of the latter intersects with
\(Z_j\), it is clear that \(T_j\) is connected.

(ii) Let \(x \in X\).  Assume \(x \in Z\), say \(x \in Z_j\).  Then it is
clear that \(x \in T_j\).  Assume \(x \notin Z\), say \(x \in X_i\).  Take
any \(j\) such that \(X_i\cap Z_j \neq \emptyset\) (such \(j\) exists;
otherwise \(X_i\) is a regular component and we have \(n=1\) and
\(m=0\)), then it is clear \(x \in T_j\).

(iii) is obvious.

(iv) We claim that we can find a permutation \((j_1,\ldots,j_m)\) of
\(\{1,\ldots,m\}\) such that \(T_{j_1} \cup \cdots\cup T_{j_r}\) is
connected for all \(1\leq r\leq m\).  Then the desired result follows
by taking \(k=j_m\).   To prove the claim, we pick \(j_1\)
arbitrarily.  Suppose that we have picked \(j_1,\ldots,j_r\) with \(r<
m\), then it suffices to take \(j_{r+1}\) such that \((T_{j_1}\cup
\cdots\cup T_{j_r}) \cap T_{j_{r+1}}\neq\emptyset\).  Indeed if such
\(j_{r+1}\) could not be found, we would deduce that
\(T_{j_1}\cup\cdots\cup T_{j_r}\) is open and closed in the connected
scheme \(X\) by (ii).  This proves the claim.
\end{proof}

Using this lemma we can reduce the computation of \(\pi_1^\pet(X)\) to
the case where the singular locus has fewer connected components.

\begin{thm}\label{thm:disconnected-covers}
  Assume that \(T_1':=T_2\cup\cdots\cup T_m\) is connected, and let
  \(D_1,\dots,D_s\) be the connected components of the regular scheme
  \(T_1\cap T_1'\).  Then the pullback functor
  \[
    \Cov(X)\longrightarrow
    \Cov(T_1)\prod_{\prod_{t=1}^s\Cov(D_t)}\Cov(T_1')
  \]
  is an equivalence.
\end{thm}
\begin{proof}
  Apply Lemma~\ref{closed van kampen} to the open cover
  \(X=T_1\cup T_1'\).  The intersection \(T_1\cap T_1'=(T_1\setminus Z)\cap
  (T_1'\setminus Z)\) is the disjoint
  union $\sqcup_{i\in S} (X_i\setminus Z)$, where
  \[S\coloneqq\set{i|X_i\cap Z_j\neq\emptyset, \,\text{ for }j=1\text{
  and some $j\geq 2$ }}\] hence its geometric covers are completely
  determined by their restrictions to the connected components.
\end{proof}

Choosing base points \(d_t\in D_t\) and identifying the fibre functors,
we obtain the following van Kampen formula.

\begin{thm}\label{thm:disconnected-structure}
  There is an isomorphism
  \[
    \pi_1^\pet(X,d_1)\simeq
    {\bf VK}\bigl(
      \pi_1^\pet(T_1,d_1),\,
      \pi_1^\pet(T_1',d_1);\,
      \pi_1^\et(D_1,d_1),\dots,\pi_1^\et(D_s,d_s)
    \bigr).
  \]
\end{thm}
\begin{proof}
  Immediate from Theorem~\ref{thm:disconnected-covers} and
  Proposition~\ref{van kampen machine}.
\end{proof}

\begin{rmk}
  The singular locus of \(T_1\) is \(Z_1\), and that of \(T_1'\) is
  \(Z_2\cup\cdots\cup Z_m\).  Hence Theorem~\ref{thm:disconnected-structure}
  provides an inductive procedure to compute \(\pi_1^\pet(X)\) in
  terms of schemes with connected singular locus, which are already
  handled by Theorem~\ref{thm:main-structure}.  By further induction
  on the dimension of the singular locus one can express
  \(\pi_1^\pet(X)\) entirely in terms of \'etale fundamental groups of
  normal schemes and discrete free groups.
\end{rmk}

\begin{cor} Let \(\scrC\) be the smallest class of Noohi groups such
  that
  \begin{itemize}
  \item \(\scrC\) contains the \'etale fundamental groups of connected
    normal
    schemes.
  \item \(\scrC\) contains discrete free groups of finite rank.
  \item \(\scrC\) is closed under fiber coproducts and quotients.
  \end{itemize}
 Then the pro-\'etale fundamental group of a Nagata J-2 connected scheme lies in \(\scrC\).
\end{cor}

\subsection{Example: zero‑dimensional singularities}

Assume that the singular locus \(Z\) is a disjoint union of spectra of
separably closed fields.  Then every connected component
of \(Z\) and of \(\tilde Z\) has trivial pro‑\'etale fundamental
group.  Let \(m\) be the number of connected components of \(Z\) and
\(\tilde m\) that of \(\tilde Z\).

\begin{thm}\label{thm:curve-case}
  Under the above hypothesis,
  \[
    \pi_1^\pet(X,x)\;\simeq\;
    \Bigl(\coprod_{i=1}^n \pi_1^\et(\tilde X_i,x_i)\Bigr)
    \;\coprod\; \mathbb{F}_{\tilde m - m - n + 1},
  \]
  where \(\mathbb{F}_r\) denotes the free discrete group of rank \(r\)
  and the coproduct is taken in the category of Noohi groups.
\end{thm}

\begin{proof}
  We argue by induction on \(m\), the number of connected components
  of the singular locus.
  Base case \(m=1\).
  Apply Theorem~\ref{thm:main-structure} (the structure theorem for
  connected singular locus).  Because every component of \(Z\) and of
  \(\tilde Z\) has trivial pro‑\'etale fundamental group, the
  VK-presentation collapses: for each irreducible
  component \(X_i\) with \(n_i\) preimages in \(\tilde Z\), the
  VK-group simplifies to
  \(\pi_1^\et(\tilde X_i) \coprod \mathbb{F}_{n_i-1}\).  Taking the
  Noohi coproduct over all \(i\) and using \(\sum_i n_i = \tilde m\)
  yields
  \begin{center}
  \resizebox{13cm}{!}{\parbox{\linewidth}{\[
    \pi_1^\pet(X) \;\simeq\;
    \Bigl(\coprod_{i=1}^n \pi_1^\et(\tilde X_i)\Bigr)
    \;\coprod\; \mathbb{F}_{\tilde m - n}.
\]}}\end{center}
  Since \(m=1\), \(\tilde m - n = \tilde m - m - n + 1\), which proves
  the formula.

  Induction step.
  Assume \(m>1\) and that the statement holds for all schemes whose
  singular locus has fewer than \(m\) connected components.
  By Lemma~\ref{lem:devissage}(iv) we may suppose, after renumbering,
  that \(T_1' = T_2 \cup \dots \cup T_m\) is connected.
  The singular locus of \(T_1\) is \(Z_1\) (one component) and that of
  \(T_1'\) is \(Z_2 \cup \dots \cup Z_m\) (\(m-1\) components).
  The intersection \(T_1 \cap T_1'\) is a disjoint union of regular
  schemes \(D_1,\dots,D_d\); each of them is an open subscheme of
  exactly one irreducible component \(X_i\) that meets both \(Z_1\) and
  some other \(Z_j\), hence \(d = \# S\) where
  \(S = S_1 \cap S_2\) with \(S_1\) the set of components meeting
  \(Z_1\) and \(S_2\) those meeting \(Z_2\cup\dots\cup Z_m\).

  By the induction hypothesis we have
  \begin{center}
\resizebox{13cm}{!}{\parbox{\linewidth}{
  \[
    \pi_1^\pet(T_1) \;\simeq\;
    \Bigl(\coprod_{i\in S_1} \pi_1^\et(\tilde X_i)\Bigr)
    \;\coprod\; \mathbb{F}_{\tilde m_1 - \# S_1},
  \]
  \[
    \pi_1^\pet(T_1') \;\simeq\;
    \Bigl(\coprod_{i\in S_2} \pi_1^\et(\tilde X_i)\Bigr)
    \;\coprod\; \mathbb{F}_{\tilde m_2 - (m-1) - \# S_2 + 1},
\]}}\end{center}
  where \(\tilde m_1\) (resp. \(\tilde m_2\)) is the number of
  connected components of \(\tilde Z\) lying over \(Z_1\) (resp. over
  \(Z_2\cup\dots\cup Z_m\)).  Clearly \(\tilde m_1+\tilde m_2=\tilde m\)
  and \(\# S_1 + \# S_2 = n + d\).

  Theorem~\ref{thm:disconnected-structure} presents
  \resizebox{1.4cm}{!}{\(\pi_1^\pet(X)\)}
  as the VK-group
  \resizebox{5.5cm}{!}{\(\mathbf{VK}(\pi_1^\pet(T_1),\pi_1^\pet(T_1');\pi_1^\et(D_1),\dots)\)}.
  Unfolding the definition, this is the Noohi coproduct of
  \(\pi_1^\pet(T_1)\), \(\pi_1^\pet(T_1')\) and a free group
  \(\mathbb{F}_{d-1}\), quotiented by the relations
  \(\psi_j(a) = v_j^{-1} \phi_j(a) v_j\) for \(j=1,\dots,d\), where
  \(\psi_j\) and \(\phi_j\) are induced by the inclusions
  \(D_j \hookrightarrow T_1\) and \(D_j \hookrightarrow T_1'\).

  For a component \(X_i\) with \(i\in S\), 
  \(D_j\) is a dense open of \(\tilde X_i\); therefore the induced map
  \(\pi_1^\et(D_j) \to \pi_1^\et(\tilde X_i)\) is surjective
  (cf.~\cite[\href{https://stacks.math.columbia.edu/tag/0BQM}{0BQM}]{stacks-project}).
  Consequently the two copies of \(\pi_1^\et(\tilde X_i)\) that appear
  in \(\pi_1^\pet(T_1)\) and \(\pi_1^\pet(T_1')\) are identified
  through the fibered coproduct over \(\pi_1^\et(D_j)\), and the free
  generator \(v_j\) does not introduce any additional relation inside
  those groups.  The remaining components (those in \(S_1 \setminus S\)
  and \(S_2 \setminus S\)) stay distinct.

  Hence the whole construction collapses to the free product of the
  distinct \(\pi_1^\et(\tilde X_i)\) for \(i=1,\dots,n\), together with
  the free summands coming from the induction hypothesis as well as the
  additional free group \(\mathbb{F}_{d-1}\).  Summing the ranks,
  \begin{center}
  \resizebox{12cm}{!}{\parbox{\textwidth}{\begin{align*}
    &\quad \bigl(\tilde m_1 - \# S_1\bigr)
      + \bigl(\tilde m_2 - (m-1) - \# S_2 + 1\bigr)
      + (d-1) \\
    &= \tilde m - m + 1 - (\# S_1 + \# S_2) + d \\
    &= \tilde m - m + 1 - (n + d) + d
      = \tilde m - m - n + 1.
  \end{align*}}}\end{center}
  This completes the induction.
\end{proof}

This generalises \cite[Theorem~1.17]{Elena18}, which treats the stable
case of a semi‑stable curve.

\def \pipet {\pi_1^\pet}
\def \piet {\pi_1^\et}
\def \bx {x}

\appendix{}
\section[Comparison with other van Kampen formulas]{Comparison with other van Kampen formulas\\ \normalfont\textit{by Marcin Lara}}

The goal of this appendix is to explain the relation of the formula presented
in this article to another general version of a van Kampen formula for $\pipet$
that appeared in \cite{Lara24} (and \cite{Lara19}).

This latter result follows the presentation of \cite{stix2006}, generalizing it from the profinite group $\piet$ to the Noohi group $\pipet$. In its most general form, \cite[Corollary 3.19]{Lara24}, it gives, for any morphism of effective descent $f \colon X' \to X$, a formula for $\pipet(X,\bx)$ in terms of fundamental groups of connected components of $X'$, a discrete fundamental group of a certain ``dual graph'', and relations coming from further fiber products. 

 Sometimes, this extra flexibility is useful, e.g. taking $f$ to be an
alteration (in the sense of de Jong) or an \'etale cover. But, as the groups
$\pipet(X,\bx)$ and $\piet(X,\bx)$ match for normal $X$, the most relevant case
for capturing the ``novel phenomena'' coming from working with $\pipet$ is when
$f \colon X' \to X$ is the normalization morphism. In this case, the formula
takes the more concrete form spelled out in \cite[Remark 3.21]{Lara24}. In
this form, the formulas (the one in \emph{loc.\ cit.\ }and the one described in
Theorem \ref{thmI} and \ref{thmII} of the present article) become closer to each other.

The main difference is the following: when the assumptions of Theorem \ref{thmII} are satisfied, \cite[Remark 3.21]{Lara24} often produces extra generators in the discrete part. Those are then killed by the relations (of course, as the resulting group is the same), but this still makes the computation somewhat cumbersome. For actually \emph{computing} the pro-\'etale fundamental group (in terms of some $\piet$'s and a discrete group), the formula in Theorem \ref{thmII} is much more straightforward.

This already manifests in one of the simplest examples of non-normal (and
non-geometri\-cally unibranch) schemes, namely the nodal curve, by which we mean
the projective line $\mathbb{P}^1_k$ over a field $k$ with two rational points
(say $0$ and $1$) glued. When $k = k^{\mathrm{sep}}$, the formula in Theorem
\ref{thmII} immediately gives the result (i.e.\ $\pipet(X,\bx) = \mathbb{Z}$),
while the other formula already requires some work (\cite[Example
3.24]{Lara24}). The ``magic'' behind the efficiency of this article's formula
sits in Lemma \ref{closed van kampen v2}, which was not present in any form in \cite{Lara24} or \cite{stix2006}.

Let us finish by mentioning that in some cases, the difference becomes less stark. When the normalizations of irreducible components inject into the scheme (e.g.\ the scheme obtained by gluing two smooth curves at one or two different rational points), one can use \cite[Proposition 3.22 + Observation 3.23]{Lara24} to efficiently compute $\pipet$. The nodal curve, however, does not have this property.

The more ``redundant'' formulation of the van Kampen formula works under more
permissive assumptions (Nagata suffices) and is still useful in ``abstract''
applications, where the precise calculation of $\pi_1$ is not strictly needed,
e.g.\ the fundamental exact sequence, K{\" u}nneth formula, or finite
presentation of fundamental groups; see \cite{Lara24} and \cite{LSS24}.

\section*{Acknowledgements}
%The first named author’s work was part of the project “Kapibara”
%supported by the European Research Council (ERC) under the European
%Union’s Horizon 2020 research and innovation programme (grant agreement No.~802787).  
The second named author  is supported
by Guangdong Basic and Applied Basic Research Foundation grant 
2025A1515012175.
\printbibliography

\end{document}